\documentclass[a4paper,10pt]{article}
\usepackage[utf8]{inputenc}
\usepackage{amsmath}
\usepackage{amsthm}
\usepackage{amsfonts}
\usepackage{amssymb}
\usepackage{mathrsfs}
\usepackage{geometry}
\usepackage{tikz}
\usetikzlibrary{automata,positioning}
\title{Coefficient Extraction Formula and Furstenberg's Theorems}
\author{YINING HU \\
CNRS, Institut de Math\'ematiques de Jussieu-PRG \\
Universit\'e Pierre et Marie Curie, Case 247 \\
4 Place Jussieu \\
F-75252 Paris Cedex 05 (France) \\
{\tt yining.hu@imj-prg.fr}}
\date{}
\begin{document}

\maketitle

\begin{abstract}
In this article, using a Proposition of Furstenberg, we give a coefficient extraction formula for algebraic series that is valid for all fields, 
of which the Flajolet-Soria coefficient extraction formula for the complex field is a special case.

\end{abstract}

\section{Introduction}
Combinatorists often use the ``Flajolet-Soria'' formula, first published in \cite{soria},  that gives an explicit expression for the coefficients of an algebraic power series.
Our main theorem is that (a generalization of) this result can be deduced from a theorem of Furstenberg going back to 1967 \cite [Proposition 2] {fu}.

In the paper of Furstenberg, a useful notion for the study of multiple variable formal power series is their diagonals. 
For the formal power series in $\kappa ((x_1,...,x_m)) $
$$f(x_1,x_2,...,x_m)=\sum_{n_i>-\mu}a_{n_1n_2...n_m}x_1^{n_1}x_2^{n_2}...x_m^{n_m}$$
its (principal) diagonal $\mathscr{D}f(t)$ is defined as the element in $\kappa((t)) $
$$\mathscr{D}f(t)=\sum a_{nn...n}t^n.$$

Furstenberg \cite{fu} proved the following results:

\newtheorem {theorem}{Theorem}
\newtheorem {prop}{Proposition}
\begin{theorem}\emph{(Furstenberg)}
\label{furst1}
Let $\kappa$ be a field of positive characteristic. Let $f(x_1,...,x_m)$ be an element of $\kappa ((x_1,...,x_m)) \cap \kappa (x_1,...,x_m)$, that is, 
$f$ is a formal power series of several variables that represents a rational function. Then $\mathscr{D}f(t) $, the diagonal of $f$, is algebraic over $\kappa(t)$.
\end{theorem}

\begin{prop}
\label{furst}
\emph{(Furstenberg)}
Let $P(X,Y)$ be a polynomial and $\varphi(X)=\sum ^\infty _1 c_n X^n$ a formal power series in $\kappa((X))$ satisfying $P(X, \varphi (X))=0$. 
If $(\partial P/\partial Y)(0,0) \neq 0$, then $$\varphi=\mathscr{D}\{Y^2 \frac{\partial P}{\partial Y}(XY,Y)/P(XY,Y) \}.$$
Here $\kappa$ is an arbitrary field.
\end{prop}

\begin{theorem}\emph{(Furstenberg)}
\label{furst2}
Let $\mathbb{F}_q$ be a finite field of characteristic $p$. If a formal power series $\phi (X)\in \mathbb{F}_q ((X))$ is algebraic over $\mathbb{F}_q(X)$, 
then $\phi=\mathscr{D}(R(X,Y))$ for a formal power series in two variables $R(X,Y)\in \mathbb{F}_q(X,Y)$ that represents a rational function of X and Y.
\end{theorem}

On the other hand, on a finite field, the formal power series algebraic over the field of rational fractions are characterized by
a simple combinatorial property:
\begin{theorem}\emph{(Christol)}
\label{christol}
Let $\mathbb{F}_q$ be a finite field of characteristic $p$. A formal power series
$$f=\sum_{n=0}^{\infty} u_n X^n\in \mathbb{F}_q[[X]] $$
is algebraic over the rational function field $\mathbb{F}_q(X)$ if and only if
the $p$-kernel of $u$ $$\{ (u_{n\cdot p^k+r})_n\;|\; k\in \mathbb{N}, r=0,1,...,p^k-1 \}$$ is finite.
\end{theorem}

These notions are revisited in this article. 
In Section 2, it is proved that the Flajolet-Soria \cite{soria} formula for coefficients of algebraic series over $\mathbb{C}(X)$ 
is a consequence of Proposition \ref{furst} of Furstenberg, with which we can obtain a similar formula that can be applied to all fields.
In Section 3, a simple proof of Theorem \ref{furst1} of Furstenberg is given using a generalization of Christol's Theorem. On the other hand, Theorem \ref{furst2}
of Furstenberg gives another proof of one direction of Christol's Theorem.
In Section 4, in order to construct a rational function whose diagonal is a given algebraic function, an algorithm is given, which calculates an annihilating polynomial 
of an algebraic function. Finally a method of finding (in $\mathbb{F}_q[[X]]$) roots of polynomials in $\mathbb{F}_q(X)[Y]$ using automata is illustrated with examples.

Among several papers related to diagonals of multivariate
formal power series, we would like to cite two recent works:
a paper by Adamczewski and Bell \cite{Adam-Bell} about a
quantitative version of the theorem of Furstenberg, and the
doctoral thesis of Lairez that gives, in particular, an interesting
characterization of binomial sums in terms of diagonals
\cite[Ch.~III]{Lairez}. 

\section{Coefficients of an algebraic series }
There is a link between combinatorial objects and power series in $\mathbb{C}$. The study of generating functions using the tools of
complex analysis gives us information about a combinatorial structure. The following theorem from \cite{soria} (see also \cite{ba} and \cite{sokal})
allows us to extract the coefficients of an algebraic series in $\mathbb{C}$ from an annihilating polynomial of it. For more reference on the formula, one can look at the
article by Banderier and Drmota\cite{ba}.
\\

 For an element $A$ in $\mathbb{F}_q((X,Y))$ 
$ A=\sum_{m,n}a_{m,n}X^mY^n$, we let  $[X^mY^n]A$ denote the coefficient $a_{m,n}$.

\begin{theorem}
\emph{(The Flajolet–Soria formula for coefficients of algebraic series)}
\label{soria}
Let $P(X,Y)$ be a polynomial over the complex field such that $P(0,0)=0$ and $P'_Y(0,0)=0$. 
The coefficients of the algebraic series $f(X)=\sum f_n X^n$, defined implicitly by $f(X)=P(X,f(X))$, 
have the form of an infinite sum $$f_n = \sum _{m\geq 1} [X^n Y^{m-1}](1-P'_Y(X,Y))P^m(X,Y).$$
\end{theorem}

Here we prove that a generalization of the formula is a simple consequence of Proposition \ref{furst}

\begin{theorem}\label{main}
Let $P(X,Y)$ be a polynomial over a field $\kappa$ such that $P(0,0)=0$ and $P'_Y(0,0)=0$. 
The coefficients of the algebraic series $f(X)=\sum _{n\geq 1} f_n X^n$, defined implicitly by $f(X)=P(X,f(X))$, 
have the form of an infinite sum  $$f_n = \sum _{m\geq 1} [X^n Y^{m-1}](1-P'_Y(X,Y))P^m(X,Y).$$
\end{theorem}

\begin{proof}
Let the polynomial $Q(X,Y)$ be defined as $Q(X,Y)=P(X,Y)-Y$, then $Q'_Y(0,0)=P'_Y(0,0)-1\neq 0$, and $Q(X, f(X))=P(X,f(X))-f(X)=0$. 
According to Proposition \ref{furst}, 
\begin{align*}f &= \mathscr{D}\{Y^2 \frac{\partial Q}{\partial Y}(XY,Y)/Q(XY,Y) \} 
\\ &=\mathscr{D}\{Y ^2 (\frac{\partial P}{\partial Y}(XY,Y)-1 )/(P(XY,Y)-Y) \}
\\&=\mathscr{D}\{Y (1-\frac{\partial P}{\partial Y}(XY,Y) )/(1-\frac{P(XY,Y)}{Y}) \}
\\&=\mathscr{D}\{Y (1-\frac{\partial P}{\partial Y}(XY,Y) )(1+\sum _{m\geq 1} (\frac{P(XY,Y)}{Y})^m) \}.
\end{align*} 
We have the last equality due to the fact that $P'_Y(0,0)=0$, $\frac{P(XY,Y)}{Y}$ has no constant term, 
and therefore $1/(1-\frac{P(XY,Y)}{Y})=1+\sum_{m\geq 1}(\frac{P(XY,Y)}{Y})^m$.
\\As in each term of $Y(1-\frac{\partial P}{\partial Y}(XY,Y))$ the power of $Y$ is larger than that of $X$, 
it cannot contribute to the diagonal.
Therefore, 

\begin{align*}
f_n & =[X^n Y^n] Y (1-\frac{\partial P}{\partial Y}(XY,Y) )(1+\sum_{m\geq 1}(\frac{P(XY,Y)}{Y})^m)
\\ &=[X^n Y^n] Y (1-\frac{\partial P}{\partial Y}(XY,Y) )(\sum_{m\geq 1}(\frac{P(XY,Y)}{Y})^m)
\\ &=\sum_{m\geq 1}[X^n Y^{m-1}] (1-\frac{\partial P(X,Y)}{\partial Y})P(X,Y)^m.
\end{align*}

\end{proof}

\section{$p$-kernel of a rational fraction and Christol's Theorem}

The following theorem is a generalization of Christol's Theorem in two variables:

\begin{theorem}\emph{(Salon \cite{salon})}
\label{salon}
Let $\mathbb{F}_q$ be a finite field of characteristic $p$. A formal power series
$$f=\sum_{(n_1,...,n_m)\in \mathbb{N}^m} u_{n_1,...,n_m} X_1^{n_1}\cdots X_m^{n_m}\in \mathbb{F}_q[[X_1,...,X_m]] $$
is algebraic over the rational function field $\mathbb{F}_q(X_1,...,X_m)$ if and only if
the $p$-kernel of $u$ $$\{ (u_{n_1\cdot p^k+r_1,..., n_m\cdot p^k+r_n})_{n_1,...,n_m}\;|\; k\in \mathbb{N}, r_i=0,1,...,p^k-1 \}$$ is finite.
\end{theorem}

This theorem, whose proof does not use Theorem \ref{furst1}, gives another proof of the latter:

\begin{proof}[Proof of Theorem \ref{furst1}]
Let $\mathbb{F}_q$ be a finite field of characteristic $p$. A formal power series in $\mathbb{F}_q[[X_1,X_m]]$
$$f(X_1,...,X_m)=\sum_{n_1,...,n_m} u_{n_1,...,n_m} X_1 ^{n_1} \cdots X_m ^{n_m}$$
that represents a rational functions of X and Y is algebraic over $\mathbb{F}_q(X_1,...,X_m)$. By Theorem \ref{salon}, the $p$-kernel
of the sequence $(u_{n_1,...,n_m})_{n_1,...,n_m}$ is finite. This means that the $p$-kernel of its diagonal $(u_{n,...,n})_n$ is finite. 
Thus, by Theorem \ref{christol}, the power series $\sum_n u_{n,...,n} X^n$ in $\mathbb{F}_q[[X]]$ is algebraic over $\mathbb{F}_q(X)$.
\end{proof}

On the other hand, by Theorem \ref{furst2} and a direct examination of the $p$-kernel, we can re-prove one direction of 
Christol's Theorem. It should be noted that Christol \cite{christol} gave a similar proof for characteristic sequences.

\begin{prop}
\label{bun}
For $P(X,Y), Q(X,Y)$ polynomials in $\mathbb{F}_q[X,Y]$, where $Q(X,Y)$ has a non-zero constant term, 
the $p$-kernel of the coefficient sequence of the formal power series $\frac{P(X,Y)}{Q(X,Y)}$ is finite.
\end{prop}

\newtheorem{coro}{Corollary}
\begin{coro}
 If the formal power series
$$f=\sum_{n=0}^{\infty} u_n X^n\in \mathbb{F}_q[[X]] $$
is algebraic over the rational function field $\mathbb{F}_q(X)$, the $p$-kernel of $u$ $$\{ (u_{n\cdot p^k+r})_n\;|\; k\in \mathbb{N}, r=0,1,...,p^k-1 \}$$ is finite.
\end{coro}
\begin{proof}
  If the formal power series $f$ 
is algebraic over the rational function field $\mathbb{F}_q(X)$, by Theorem \ref{furst2}, $f$ is the diagonal of a rational function 
$R(X,Y)$ in $\mathbb{F}_q[[X,Y]]$. By Proposition \ref{bun}, the $p$-kernel of $R(X,Y)$ is finite, and therefore the $p$-kernel of its 
diagonal $f$ is finite.
\end{proof}

Before proving Proposition \ref{bun} we first recall an easy lemma:

\newtheorem{lem}{Lemma}
\begin{lem}
 For an element $A$ in $\mathbb{F}_q((X,Y))$, 
$ A=\sum_{m,n}a_{m,n}X^mY^n$, we let $\Lambda_{r,s}(A)$ denote the formal power series $\sum_{m,n}a_{mq+r,nq+s}X^mY^n$. Note that $\Lambda_{r,s}$ is sometimes called
a Cartier operator.
For $A,B \in \mathbb{F}_q((X,Y)) $, we have $\Lambda_{r,s}(A^q B)=A\Lambda_{r,s}(B)$
\end{lem}
\begin{proof}
 As all the coefficients are in $\mathbb{F}_q$, we have $$ (\sum_{m,n}  a_{m,n} x^m y^n)^q 
 = \sum _{m,n} a^q_{m,n} x^{mq} y^{nq} =\sum _{m,n} a_{m,n} x^{mq} y^{nq}$$

 \begin{align*} [x^m y^n]\Lambda_{r,s}(A^q B) &= [x^{mq+r} y^{nq+s}](A^q B)\\
&=\sum\limits_{\substack{c_1+c_2=mq+r \\c_3+c_4 =nq+s}} [x^{c_1} y^{c_3}]A^q \cdot [x^{c_2} y^{c_4}] B \\
&=\sum\limits_{\substack{q c'_1+c_2=mq+r \\qc'_3+c_4 =nq+s}} a_{c'_1 c'_3} \cdot b_{c_2 c_4}\\
&=\sum\limits_{\substack{c'_1+c'_2=m \\c'_3+c'_4 =n}} a_{c'_1 c'_3} \cdot b_{c'_2 q +r, c'_4 q + s}\\
&=[x^m y^n](A \Lambda_{r,s}(B)).
\end{align*}
\end{proof}

\begin{proof}[Proof of Proposition \ref{bun}]
We let $\mathbf{\Lambda}$ denote the set of operators $\{ \Lambda_{r,s} \;|\;r,s\in\{0,1,...,q-1\}$.
The $q$-kernel of $\frac{P}{Q}$ is generated by $\frac{P}{Q}$ 
and the operations in $\mathbf{\Lambda}$. 
\\Let $\Lambda_0$ be an element in $\mathbf{\Lambda}$, then by the previous lemma,
$$\Lambda_0(\frac{P}{Q})=\Lambda_0(\frac{PQ^{q-1}}{Q^p})=\frac{\Lambda_0(PQ^{p-1})}{Q}.$$
If we let $a$ and $b$ denote the degree of $P$ and $Q$, then 
$$\textnormal{deg}(\Lambda_0(PQ^{q-1})) \leq \frac{a+b(q-1)}{q} < a+b.$$
Let $P_1$ denote $\Lambda_0(PQ^{q-1})$. Take another element $\Lambda_1$ in $\mathbf{\Lambda}$, then
$$\Lambda_1(\Lambda_0(\frac{P}{Q}))=\Lambda_1(\frac{P_1}{Q})=\frac{\Lambda_1(P_1Q^{q-1})}{Q}$$
$$\textnormal{deg}(\Lambda_1(P_1Q^{q-1}))\leq(\textnormal{deg}(P_1)+\textnormal{deg}(Q)\cdot (q-1))/q < (a+b+b\cdot (q-1))/q < a+b.$$
By induction, after applying the elements of $\mathbf{\Lambda}$ to $\frac{P}{Q}$, we always get an element in the set
$$\{\frac{R}{Q}\;|\; R\in \mathbb{F}_q(X,Y),\; \textnormal{deg}(R)<a+b \}.$$
As the field is finite, the number of polynomials of degree less than $a+b$ is finite and so is the $q$-kernel. 
This ends the proof as the $p-$kernel of a sequence is finite if and only if its $q-$kernel is finite.

\end{proof}

\section{Miscellaneous}
\subsection{Annihilating polynomial of algebraic functions}
Theorem \ref{furst2} states that if a series $\phi(X)=\sum\limits_{n\geq 0} u_n X^n \in \mathbb{F}_q[[X]]$ is algebraic over $\mathbb{F}_q(X)$,
then it is the diagonal of a rational function
$R(X,Y)\in \mathbb{F}_q(X,Y)$. The same result is proved by Fagnot \cite{fagnot} using combinatorial methods. 

From the proof of Proposition 1 and Theorem 2 in \cite{fu} we know that $R(X,Y)$ can be constructed from a relation of the form 
\begin{equation} A_0(X)\phi^{q^l}(X)+A_1(X)\phi^{q^{l+1}}(X)+ \cdots + A_n(X)\phi^{q^{n+l}}(x)=0. \label{relation} \end{equation}

We now show how to obtain such a relation for a given algebraic series $\phi(X)=\sum\limits_{n\geq 0} u_n X^n \in \mathbb{F}_q[[X]]$.
Suppose that the sequence $(u_n)_n$ is given. By this we mean that we know either an automaton that generates $(u_n)_n$ or the $q$-kernel of $(u_n)_n$ with initial conditions. 
For the definition of automata and its link with algebraic series see \cite{automate}.

Let 
 $E=\{ u^1,u^2,...,u^d\}$ with $u^1=u$ be the $q$-kernel of $(u_n)_n$. $E$ is stable by the maps:
$$(u^i_{n})_n \rightarrow (u^i_{qn+r})_n\;\; \textnormal{for}\;\; r=0,1,...,p-1 .$$ 
That is to say, there exists a map $f$ from $\{1,2,...,d\}\times \{0,1,...,p-1\}$ to $ \{1,2,...,d\}$ such that
$$(u^i_{nq+r})_n=(u^{f(i,r)}_{n})_n$$
Define the matrix $A(X)$ in $\mathbb{F}_q[X]^{d\times d}$, where
$$
A_{i,j}(X) = \sum_{\{r|f(i,r)=j\}} X^r.
$$
Define the formal power series $G_1,...,G_d$ by $$G_i(X)=\sum_{n=0}^{\infty}u^i_{n}X^n .$$
The series can be written as \begin{align*}G_i(X)&=\sum_{r=0}^{q-1}\sum_{m=0}^{\infty} u^i_{qm+r}X^{qm+r}
\\&=\sum_{r=0}^{q-1}X^r\sum_{m=0}^{\infty}u^i_{qm+r}X^{qm},
\\&=\sum_{j}A_{i,j}(X)\sum_{m=0}^{\infty}u^j_{m}X^{qm},
\\&=\sum_{j}A_{i,j}(X) G^j(X^q).
\end{align*}

Writing the equalities in matrix form and using the fact that $P(X^{q^k})=P(X)^{q^k}$ for $P(X)$ in $\mathbb{F}_q[X]$ we have:
$$\begin{pmatrix}
   G_1(X)
   \\ G_2(X)
   \\ \vdots
   \\ G_d(X)
  \end{pmatrix}=
  A(X)\begin{pmatrix}
   G_1(X)^q
   \\ G_2(X)^q
   \\ \vdots
   \\ G_d(X)^q
  \end{pmatrix}.$$

More generally, 
$$\begin{pmatrix}
   G_1(X)^{q^k}
   \\ G_2(X)^{q^k}
   \\ \vdots
   \\ G_d(X)^{q^k}
  \end{pmatrix}=
  A(X^{p^k})\begin{pmatrix}
   G_1(X)^{q^{k+1}}
   \\ G_2(X)^{q^{k+1}}
   \\ \vdots
   \\ G_d(X)^{q^{k+1}}
  \end{pmatrix}.$$
 
 And therefore for $k\geq 1$
 $$\begin{pmatrix}
   G_1(X)
   \\ G_2(X)
   \\ \vdots
   \\ G_d(X)
  \end{pmatrix}
  =\prod_{i=0}^{k-1}A(X^{q^i})
  \begin{pmatrix}
   
   G_1(X)^{q^k}
   \\ G_2(X)^{q^k}
   \\ \vdots
   \\ G_d(X)^{q^k}
  \end{pmatrix}.$$
 Denoting the $i$-th row of a matrix $M$ by $M_i$, we have
  $$\begin{pmatrix}
   G_1(X)
   \\ G_1(X)^q
   \\ \vdots
   \\ G_1(X)^{q^d}
  \end{pmatrix}
  =
  \begin{pmatrix}
   (\prod_{i=0}^{d}A(X^{q^i}))_1
   \\(\prod_{i=1}^{d}A(X^{q^i}))_1
   \\ \vdots
   \\ (\prod_{i=d}^{d}A(X^{q^i}))_1
  \end{pmatrix} 
  \begin{pmatrix}
   
   G_1(X)^{q^{d+1}}
   \\ G_2(X)^{q^{d+1}}
   \\ \vdots
   \\ G_d(X)^{q^{d+1}}
  \end{pmatrix}.$$
  This equality is of the form 
   $$\begin{pmatrix}
   G_1(X)
   \\ G_1(X)^q
   \\ \vdots
   \\ G_1(X)^{q^{d}}
  \end{pmatrix}
  =B(X) \cdot
  \begin{pmatrix}
   
   G_1(X)^{q^{d+1}}
   \\ G_2(X)^{q^{d+1}}
   \\ \vdots
   \\ G_d(X)^{q^{d+1}}
  \end{pmatrix},$$
  where $B(X)$ is a matrix in $\mathbb{F}_q(X)^{(d+1)\times d}$. There exists a linear combination of the $d+1$ rows of $B(X)$ that
  is equal to $0$. This means that the linear combination of $G_1(X),...,G_1(X)^{q^d}$ with the same coefficients is $0$, which is a relation
  of the form \eqref{relation}.
 
\subsection{Formal power series solutions of polynomials}
The main result of this article (Theorem \ref{main}) allows us to calculate the coefficients of an algebraic series using an annihilating polynomial
of it, when the latter satisfies certain conditions. For a formal power series in $\mathbb{F}_q[[X]]$ that is algebraic over $\mathbb{F}_q(X)$, we can
calculate its coefficients in a simpler way by knowing an automaton that generates them. 
We have shown how to find an annihilating polynomial of an algebraic function in the previous subsection. 
Now we start from a polynomial in $\mathbb{F}_q (X)[Y]$ and show how to find automata that generate its roots in $\mathbb{F}[[X]]$ with typical examples.
  
\theoremstyle{definition}
\newtheorem{example}{Example}
\begin{example}
$P(X,Y)=(1+X)^3 Y^2+(1+X)^2 Y + X \in \mathbb{F}_2[X,Y]$. There exists $f(X)\in \mathbb{F}_2[X]$ such that $P(X,f(X))=0$. This is because  $P(X, \sum \limits_{n\geq 0} a_n X^n)=0$
if and only if $(a_n)$ satisfies the condition
$$[x^n]P(X, \sum \limits_{n\geq 0} a_n X^n)=0,\;\;\mbox{for all } n\in \mathbb{N}.$$
The first two equations are $ a_0+a_0=0$ and $a_0+a_1+1=0$. And for $n\geq 2$, $a_n$ appears for the first time in the $n-$th equation. As we have a new variable for every 
constraint except for the first two equations, there exist two solutions, which correspond to $a_0=0$ and $a_0=1$. 

Let $f$ be a solution of $P(X, f(X))=0$, $f$ has degree 2 over $\mathbb{F}_q(X)$, therefore $f$, $f^2$ and $f^4$ are linearly dependent over $\mathbb{F}_q(X)$. A relation can be 
found by writing them all as linear combinations of $1$ and $f$. That is,

\begin{align*}f^2&=\frac{X}{(1+X)^3}+\frac{1}{1+X}f\\
f^4&=(f^2)^2\\
&=\frac{X^2}{(1+X)^6}+\frac{1}{(1+X)^2} f^2 \\
&=\frac{X^2}{(1+X)^6}+ \frac{1}{(1+X)^2}\Big(\frac{X}{(1+X)^3}+\frac{1}{1+X}f\Big)\\
&=\frac{X}{(1+X)^6}+\frac{1}{(1+X)^3}f
\end{align*}

Therefore, $$f^4+\frac{1}{(1+X)^3}f^2+\frac{X}{(1+X)^4}f=0.$$
We apply the Cartier operators $\Lambda_0$ and $\Lambda_1$ repetitively to
$$f=\frac{(1+X)^4}{X}f^4+\frac{1+X}{X} f^2. $$

We have
\begin{align*}
\Lambda_0(f)&=f\\
\Lambda_1(f)&=\frac{f^2}{x}+xf^2+\frac{f}{x}=:f_1\\
\Lambda_0(f_1)&=f_1\\
\Lambda_1(f_1)&=f.\\
\end{align*}

Therefore $f$ is generated by the automaton below with initial state $i$:
\begin{tikzpicture}[shorten >=1pt,node distance=2cm,on grid,auto] 
   \node[state] (i)   {$i$}; 
   \node[state] (a) [right=of i]{$a$}; 

    \path[->] 
    (i) edge [loop above] node  {0} (i)
          edge [bend left] node  {1} (a)
    (a) edge [loop above] node  {0} (a)
          edge [bend left] node  {1} (i);

\end{tikzpicture}.

There are four sequences that can be defined by this automaton, and the only two that satisfy the condition $a_0+a_1=1=0$ are the sequences that correspond to the maps 
$$\pi_1(i)=0$$
$$\pi_1(a)=1,$$
and
$$\pi_2(i)=1$$
$$\pi_2(a)=0.$$

These two sequences are the solutions of $P(X, f(X))=0$, because we know from the beginning of the example that the equation $P(X, f(X))=0$ 
has exactly two solutions in $\mathbb{F}_q[X]$, and each solution is defined by the automaton above. These two sequences are the Thue-Morse sequence and 
its bitwise negation.
\end{example}

\begin{example}$Q(X,Y)=Y^2+(1+X)Y+X^2\in \mathbb{F}_2[X,Y]$. It can be shown with the same argument as for the first example, that
the equation $Q(X,f(X))=0$ has two solutions in $\mathbb{F}_q[X]$. 

Writing $f$, $f^2$ and $f^4$ as linear combinations of $1$ and $f$ over $\mathbb{F}_q(X)$, we find the relation
$$f^4+f^2+(X^2+X^3)f=0.$$

Applying $\Lambda_0$ and $\Lambda_1$ repetitively to $$f=\frac{f^4}{X^2+X^3}+\frac{f^2}{X^2+X^3}$$ we get

\begin{align*}
 \Lambda_0(f)&=\frac{f^2}{X(1+X)}+\frac{f}{X(1+X}=:f_1\\
 \Lambda_1(f)&=f_1\\
 \Lambda_0(f_1)&=\frac{f}{1+X}=:f_2\\
 \Lambda_1(f_1)&=\frac{f^2}{X^2(1+X)}+\frac{f}{X^2}=:f_3\\
 \Lambda_0(f_2)&=f_1\\
 \Lambda_1(f_2)&=0\\
 \Lambda_0(f_3)&=A_1(f_3)=f_3.\\
 \end{align*}

 Therefore $f$ can be generated by the automaton below with inital state $i$:
 
\qquad  \qquad  \qquad  \qquad \qquad \qquad  
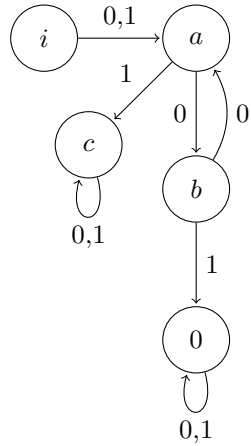
\begin{tikzpicture}[shorten >=1pt,node distance=2cm,on grid,auto] 
   \node[state] (i)   {$i$}; 
   \node[state] (a) [right=of i]{$a$}; 
   \node[state] (b) [below=of a]{$b$};
   \node[state] (c)  [below left=of a]{$c$};
   \node[state]  (0) [below=of b]{$0$};
   \path[->] 
    (i) edge node  {0,1} (a)
          
    (a) edge  node [swap] {0} (b)
          edge  node [swap] {1} (c)
    (b) edge [bend right] node [swap] {0} (a)
       edge node {1} (0)
     (c) edge [loop below] node {0,1} (c)  
     (0) edge [loop below] node {0,1} (0)  ;
\end{tikzpicture}

\end{example}
The equation $[X^1]Q(X,f(X))=0$ gives $a_0=a_1$, and the equation $[X^2]Q(X,f(X))=0$ gives $a_2=1$. The sequences that satisfy these two conditions are defined by this
automaton
with the map
$$\pi_1(i)=\pi_1(a)=0$$
$$\pi_1(c)=1 $$
$$\pi_1(0)=0 $$
or
$$\pi_2(i)=\pi_2(a)=1$$
$$\pi_2(c)=1 $$
$$\pi_2(0)=0. $$
These two sequences are the two solutions of $Q(X,f(X))=0$ in $\mathbb{F}_q[X].$

\end{document}